\documentclass[11pt,twoside,reqno,psamsfonts]{amsart}

\usepackage[left=2.9cm,top=3cm,right=2.9cm]{geometry}              
\usepackage[colorlinks = true, citecolor = black, linkcolor = black, urlcolor = black, pdfstartview=FitH]{hyperref}  
\usepackage[pdftex]{graphicx}
\usepackage[utf8]{inputenc}
\usepackage{microtype, float}
\usepackage{amsmath, verbatim}
\usepackage{amssymb}
\usepackage{epstopdf}
\usepackage{epsfig}
\usepackage{mathscinet}

\let\oldtocsection=\tocsection
\let\oldtocsubsection=\tocsubsection
\let\oldtocsubsubsection=\tocsubsubsection
\renewcommand{\tocsection}[2]{\hspace{0em}\textbf{\oldtocsection{#1}{#2}}}
\renewcommand{\tocsubsection}[2]{\hspace{1.8em}\oldtocsubsection{#1}{#2}}
\renewcommand{\tocsubsubsection}[2]{\hspace{3em}\oldtocsubsubsection{#1}{#2}}

\DeclareUnicodeCharacter{00A0}{ } 

\numberwithin{equation}{section}

\theoremstyle{plain}
\newtheorem{thm}{Theorem}[section]
\newtheorem{theorem}[thm]{Theorem}

\newtheorem{lemma}[thm]{Lemma}
\newtheorem{prop}[thm]{Proposition}

\newtheorem{question}[thm]{Question}

\theoremstyle{definition}

\theoremstyle{remark}

\newcommand{\R}{\mathbb{R}}

\newcommand{\C}{\mathbb{C}}

\newcommand{\N}{\mathbb{N}}

\renewcommand{\H}{\mathbb{H}}
\renewcommand{\S}{\mathbb{S}}

\newcommand{\cS}{\mathcal{S}}

\newcommand{\eps}{\varepsilon}

\newcommand{\Z}{\mathbb{Z}}

\renewcommand{\epsilon}{\varepsilon}
\renewcommand{\rho}{\varrho}
\renewcommand{\phi}{\varphi}

\newcommand{\1}{\mathrm{\mathbf{1}}}

\renewcommand{\hat}{\widehat}

\renewcommand{\iint}{\int\hspace{-0.1in}\int}

\newcommand{\To}{\longrightarrow} 
\newcommand{\HS}{\mathrm{HS}}
\renewcommand{\Im}{\mathrm{Im\,}}
\renewcommand{\Re}{\mathrm{Re\,}}

\newcommand{\id}{\operatorname{id}}
\renewcommand{\mod}{\,\,\mathrm{mod}\,}

\DeclareMathOperator{\Vol}{Vol}
\DeclareMathOperator{\InjRad}{InjRad}
\DeclareMathOperator{\SL}{SL}
\DeclareMathOperator{\PSL}{PSL}

 \def\Xint#1{\mathchoice
      {\XXint\displaystyle\textstyle{#1}}%
      {\XXint\textstyle\scriptstyle{#1}}%
      {\XXint\scriptstyle\scriptscriptstyle{#1}}%
      {\XXint\scriptscriptstyle\scriptscriptstyle{#1}}%
      \!\int}
   \def\XXint#1#2#3{{\setbox0=\hbox{$#1{#2#3}{\int}$}
        \vcenter{\hbox{$#2#3$}}\kern-.5\wd0}}
   
   \def\dashint{\Xint-}

\usepackage[usenames,dvipsnames]{xcolor}

\title[Quantum ergodicity and Benjamini-Schramm convergence]{Quantum ergodicity and Benjamini-Schramm convergence of hyperbolic surfaces}

\author{Etienne Le~Masson and Tuomas Sahlsten}
\address{School of Mathematics, University of Bristol, University Walk, Bristol, BS8 1TW, UK}
\email{etienne.lemasson@bristol.ac.uk}
\email{tuomas.sahlsten@bristol.ac.uk}


\thanks{E. Le Masson was partially supported by ERC advanced grants $\sharp$267259 and $\sharp$291147, the NSF grant $\sharp$0932078-000 for MSRI and Marie Sk{\l}odowska-Curie Individual Fellowship grant $\sharp$703162. T. Sahlsten was partially supported by the ERC starting grant $\sharp$306494 and Marie Sk{\l}odowska-Curie Individual Fellowship grant $\sharp$655310.}

\begin{document}

\begin{abstract}
We present a quantum ergodicity theorem for fixed spectral window and sequences of compact hyperbolic surfaces converging to the hyperbolic plane in the sense of Benjamini and Schramm. This addresses a question posed by Colin de Verdi\`{e}re. Our theorem is inspired by results for eigenfunctions on large regular graphs by Anantharaman and the first-named author. It applies in particular to eigenfunctions on compact arithmetic surfaces in the level aspect, which connects it to a question of Nelson on Maass forms. The proof is based on a wave propagation approach recently considered by Brooks, Lindenstrauss and the first-named author on discrete graphs. It does not use any microlocal analysis, making it quite different from the usual proof of quantum ergodicity in the large eigenvalue limit. Moreover, we replace the wave propagator with renormalised averaging operators over discs, which simplifies the analysis and allows us to make use of a general ergodic theorem of Nevo. As a consequence of this approach, we require little regularity on the observables.
\end{abstract}

\maketitle

\section{Introduction}\label{intro}

\subsection{Background and motivation} Let $\H$ be the hyperbolic plane, and $X = \Gamma\backslash\H$ be a compact hyperbolic surface. Let $\Delta$ be the Laplace-Beltrami operator on $X$.  We consider an arbitrary orthonormal basis $\{\psi_j\}_{j\in\N}$ of $L^2(X)$ consisting of eigenfunctions of $\Delta$, and denote by $\{\lambda_j\}_{j\in\N}$ the non-decreasing sequence of corresponding eigenvalues. The surface $X$ is known to satisfy  \emph{Quantum Ergodicity} (QE), a spectral analog of the classical ergodic property, stating that 
$$\frac{1}{N(\lambda)} \sum_{\lambda_j\leq \lambda} \left| \langle \psi_j, a \, \psi_j\rangle
- \dashint_X a\, d\Vol \right|^2 \longrightarrow 0,$$
as $\lambda \to \infty$, where $a$ is any fixed continuous observable on $X$ and the symbol $\dashint$ means that we divide the integral by the total volume.  Here $N(\lambda)=|\{\lambda_j\leq \lambda\}|$ is the eigenvalue counting function (with multiplicity).  More generally, one can replace multiplication by $a$ in the inner product by a zero-order pseudodifferential operator $A$, and the $a$ in the integral by the principal symbol $\sigma_A$ of $A$.
The quantum ergodicity property implies that there exists a density $1$ subsequence $(\psi_{j_k})$ such that  the measures $|\psi_{j_k}|^2d\Vol$ converge weakly to the uniform measure.  It was first shown to hold in this setting by Shnirelman and Zelditch \cite{Sni74,Zel87}, and generalised to any Riemannian manifold with ergodic geodesic flow by Colin de Verdi\`ere \cite{CdV85}.

Ergodicity of the geodesic flow is not enough to avoid having to extract a subsequence, as was shown by Hassell \cite{Has10}. However, for negatively curved manifolds, Rudnick and Sarnak conjectured that the full sequence of eigenfunctions should equidistribute asymptotically \cite{RS94}. This \textit{Quantum Unique Ergodicity} (QUE) conjecture was proved in the setting of compact arithmetic hyperbolic surfaces of congruence type by Lindenstrauss, for joint eigenfunctions of the Laplacian and an algebra of Hecke operators, arising from the arithmetic structure of the surface (\cite{Lin06}, see also \cite{BL14} for a strengthening of this result using only one Hecke operator). By excluding the possibility of an \emph{escape of mass} in the cusp, Soundararajan \cite{Sou10} completed Lindenstrauss's result for the non-compact case of the modular surface $\SL_{2}(\Z) \backslash \H$. An alternative approach was pursued by Anantharaman \cite{Ana08} (see also \cite{AN07}), who proved entropy bounds for eigenfunctions on Riemannian manifolds with Anosov geodesic flow, excluding in particular the possibility of concentration of all the mass on closed geodesics.

We will be interested here in a different point of view, where the eigenvalues stay bounded in a fixed interval, and we take instead a large scale geometric limit by considering sequences of hyperbolic surfaces converging to the plane. The convergence we consider is in the sense of Benjamini and Schramm, which we will define precisely.  This kind of setting for the study of eigenfunctions has attracted attention separately in the field of modular forms and on discrete graphs. Let us review some of these results to motivate our theorem.

In the field of modular forms, and in parallel to Lindenstrauss's results, a theory of QUE for holomorphic forms was developed (see \cite{Sar11} for a survey), culminating in the work of Holowinsky and Soundararajan \cite{HS10}. Note that these results are concerned with \emph{holomorphic Hecke eigenforms} and are proved for specific non-compact arithmetic surfaces $\Gamma_0(q) \backslash \H$, $q \in \N$, where
$$\Gamma_0(q) = \Big\{\begin{pmatrix} a & b \\ c & d\end{pmatrix} \in \SL_2(\Z) : c \equiv 0 \mod q\Big\},$$
that is, congruence coverings of the modular surface. In the setting of holomorphic forms, two parameters are available: the \textit{weight}, which is roughly analogous to the Laplace eigenvalue, and the \textit{level} $q$, which characterises the congruence covering. The results of Holowinsky and Soundararajan concern the limit of large weights, and it is natural to ask what happens for a fixed weight in the level aspect, that is when the level goes to infinity.\footnote{In the \emph{compact} case, it is a consequence of \cite[Theorem 5.2]{ABBGNRSPreprint} that arithmetic surfaces of congruence type converge to the plane in the sense of Benjamini and Schramm when the level goes to infinity.}
 Building on the methods of Holowinsky and Soundararajan, QUE-type theorems in the level aspect have been proved recently by Nelson \cite{Nel11}, and Nelson, Pitale and Saha \cite{NPS14}. However, some important elements are missing to adapt these methods to Maass forms, that is eigenfunctions of the Laplacian, as was noted by Nelson \cite[Remark 1.7]{Nel11}.

Independently of the theory of modular forms, there has been a growing interest in the problem of quantum chaos on discrete regular graphs since the pioneering work of Jakobson, Miller, Rivin, and Rudnick \cite{JMRR99} (see also \cite{Smi13} for a survey on the subject). Regular graphs can be seen as discrete analogues of hyperbolic surfaces. In this setting the eigenvalues of the Laplacian are bounded, and the relevant asymptotics is that of large graphs, in the sense of the number of vertices going to infinity. A first result of delocalisation of eigenfunctions was shown by Brooks and Lindenstrauss \cite{BL13}. Using a discrete version of microlocal calculus \cite{LM14}, Anantharaman and Le~Masson then proved a general quantum ergodicity result for regular graphs \cite{ALM15}. Since then, alternative proofs of this result have been proposed by Brooks, Le~Masson and Lindenstrauss \cite{BLML}, and Ananthraman \cite{A}. In the light of these results on graphs, one can ask if a quantum ergodicity theorem can be proved on manifolds for large spatial scale and fixed eigenvalues. This problem was suggested to the first-named author by Yves Colin de Verdi\`ere and we are addressing it in this paper. In the special case of compact arithmetic surfaces, our result gives in particular an equidistribution result for eigenfunctions in the level aspect, connecting it to the question of equidistribution of Maass forms.

Perhaps surprisingly, the adaptation of the pseudo-differential methods of \cite{ALM15} in the continuous setting presents some important difficulties, and the proof of our theorem will instead be inspired by the wave propagation methods proposed in \cite{BLML} for the discrete Laplacian on graphs. By introducing a continuous analogue of these methods, we are able to dispense completely with any pseudo-differential calculus, which makes our proof quite different from the usual proof of the quantum ergodicity theorem. In contrast with the latter, where continuity of the observables is usually required, we only need $L^2$, or bounded measurable test functions depending on the case considered.

\subsection{Main results} To state our result, let us define the notion of Benjamini-Schramm convergence. This notion was introduced by Benjamini and Schramm as a model of distributional convergence of graphs \cite{BS01}, but it has a natural analogue in the continuous setting, where it was used in particular by Abert \textit{et al.} \cite{ABBGNRSPreprint,ABBGNRS} and Bowen \cite{Bow15}. Let $\InjRad_X(z)$ be the injectivity radius of $X$ at $z \in X$ and $\InjRad(X) = \inf_{z \in X} \InjRad_X(z)$ the minimal injectivity radius (see Section \ref{sec:prelim} for a definition). 
We say that a sequence of compact hyperbolic surfaces $X_n = \Gamma_n \backslash \H $ \emph{Benjamini-Schramm converges} to the hyperbolic plane $\H$ if for any $R > 0$,
\begin{equation}\label{eq:limbs}\lim_{n \to \infty} \frac{\Vol(\{ z \in X_n : \InjRad_{X_n}(z) < R\})}{\Vol(X_n)} = 0.\end{equation}
The convergence means that the probability for a ball of radius $R$ at a random point in $X_n$ to be isometric to a ball in $\H$ tends to $1$ as $n \to \infty$. Note that this is automatically satisfied if $\InjRad(X_n) \to \infty$. We will discuss some examples after the statement of the results.
The theorem we prove is the following:

\begin{thm}\label{thm:sequences}
Fix a closed interval $I \subset (1/4,+\infty)$. Let $(X_n)_{n \in \N}$ be a sequence of compact connected hyperbolic surfaces, $\lambda_0^{(n)} = 0 < \lambda_1^{(n)} \leq \lambda_2^{(n)} \leq \ldots$ be the sequence of eigenvalues of the Laplacian on $X_n$, and $\{\psi_k^{(n)}\}_{k\in\N}$ a corresponding orthonormal basis of eigenfunctions in $L^2(X_n)$. 

We assume that $(X_n)$ satisfies the following conditions:
\begin{enumerate}
	\item  $(X_n)$ converges to $\H$ in the sense of Benjamini and Schramm; \label{condBS}
	\item There exists $\ell_{\min} >0$ such that for all $n\in\N$, $\InjRad(X_n) \geq \ell_{\min}$. 
	\label{condlmin}
	\item  There exists $\beta > 0$ such that for all $n\in\N$, the Laplacian on $X_n$ has a spectral gap $\beta$, 
	i.e. for every $n\in \N$, $\lambda_1^{(n)} \geq \beta$;\label{condgap}
\end{enumerate} 
 Then for any uniformly bounded sequence of measurable functions $(a_n)_{n\in\N}$, we have
\begin{equation}\label{e:QEBS}
 \frac1{N(X_n,I)} \sum_{j \,:\, \lambda_j^{(n)} \in I} \left| \langle \psi_j^{(n)}, a_n \, \psi_j^{(n)} \rangle - \dashint a_n \, d\Vol_{X_n} \right|^2 \To 0
 \end{equation}
when $n \to +\infty$, where $N(X_n,I)$ is the number of eigenvalues in the interval $I$ counted with multiplicity. Note that $N(X_n,I) \sim_I \Vol(X_n)$ by Lemma \ref{lma:boundingeigenvalues}.
\end{thm}

The theorem means that on \textit{large} hyperbolic surfaces, for any fixed spectral interval and given a test function $a$, most of the eigenfunctions of the Laplacian evaluated on $a$ approach the uniform measure. To be more precise, we need to discuss the type of sequences of test functions $a_n$ that allow such an interpretation. In order to get a true equidistribution, we want to ensure that cancellations occur between the terms $\langle \psi_j^{(n)}, a_n \, \psi_j^{(n)} \rangle$ and $\dashint a_n \, d\Vol_{X_n}$. It means that the latter quantity should not tend to $0$ too fast compared to the rate of convergence of \eqref{e:QEBS} when $n \to +\infty$. From \eqref{e:QEBS} we can deduce that cancellations will occur when the support of $a_n$ grows like $\Vol(X_n)$. For example, for any fixed constant $0< c <1$, the theorem says that we have equidistribution in a portion of the surface of area $c \Vol(X_n)$. Although the theorem allows to take observables $a_n$ with arbitrarily small support $o(\Vol(X_n))$, in this case it only says that
\begin{equation*}
 \frac1{N(X_n,I)} \sum_{j \,:\, \lambda_j^{(n)} \in I} \left| \langle \psi_j^{(n)}, a_n \, \psi_j^{(n)} \rangle\right|^2 \To 0
 \end{equation*}
 as $n \to +\infty$.
By giving a quantitative estimate of the rate of convergence of \eqref{e:QEBS},  we will see that we actually have equidistribution on smaller scales.

Theorem \ref{thm:sequences} is a consequence of the following quantitative version, together with an asymptotic estimate of the number of eigenvalues in bounded intervals under Benjamini-Schramm convergence (see Section \ref{sec:spectralBS}).

\begin{theorem}\label{thm:quantitative}
Let $X = \Gamma\backslash\H$ be a compact connected hyperbolic surface. Let $\lambda_0 \leq \lambda_1 \leq \ldots$ be the sequence of eigenvalues of the Laplacian on $X$, and $\{\psi_k\}_{k\in\N}$ a corresponding orthonormal basis of eigenfunctions in $L^2(X)$. Fix a closed interval $I \subset (1/4,+\infty)$, then there exists $R_I > 0$ such that for any $a \in L^\infty(X)$ we have
$$ \sum_{\lambda_j \in I} \left| \langle \psi_j, a \, \psi_j \rangle - \dashint a \, d\Vol_X \right|^2 \lesssim_I \frac{\|a\|_2^2}{\rho(\lambda_1)^2R} + \frac{e^{4R}}{\ell_{\min}} \Vol(\{ z \in X : \InjRad_{X}(z) < R\}) \|a\|_\infty^2$$
for all $R > R_I$. Here $\rho(\lambda_1) > 0$ is a constant depending only on the spectral gap $\lambda_1$ of the Laplacian on $X$, and $\ell_{\min}$ is the injectivity radius of $X$. 

In particular, if $R < \ell_{\min}$, we have
\begin{equation}\label{e:QErate}
 \sum_{\lambda_j \in I} \left| \langle \psi_j, a \, \psi_j \rangle - \dashint a \, d\Vol_X \right|^2 \lesssim_I \frac{\|a\|_2^2}{\rho(\lambda_1)^2R},
 \end{equation}
for any function $a \in L^2(X)$.
\end{theorem}

Here $A \lesssim_I B$ means there exist a constant $C > 0$ that only depends on $I$ with $A \leq C B$. As a consequence of this theorem, if we replace conditions \eqref{condBS} and \eqref{condlmin} in Theorem \ref{thm:sequences} with the assumption that the radius of injectivity go to infinity, then we get the same conclusion assuming only that the observables $a_n$ are in $L^2$ and that $\|a_n\|_2^2 = O(\Vol(X_n))$. For simplicity, let us comment on the rate \eqref{e:QErate} of the large injectivity radius case. Assuming $\|a_n\|_2^2 = O(\Vol(X_n))$, we obtain 
$$ \frac1{N(X_n,I)} \sum_{j \,:\, \lambda_j^{(n)} \in I} \left| \langle \psi_j^{(n)}, a_n \, \psi_j^{(n)} \rangle - \dashint a_n \, d\Vol_{X_n} \right|^2 = O\left( \frac1{R_n}\right), $$
 where $R_n$ is the injectivity radius. Note that we have $R_n \leq \log(\Vol(X_n))$. Assuming an ideal case where $R_n$ grows like $\log(\Vol(X_n))$, then the rate is analogous to the rate $\log(\lambda)^{-1}$ obtained in the large eigenvalue case $\lambda \to \infty$ by Zelditch \cite{Zel94}.
The rate in \eqref{e:QErate} also gives equidistribution on smaller scales, for observables with support of area $c\Vol(X_n)/R_n$, where $0< c <1$ is a fixed constant. In the large eigenvalue limit, the question of Quantum Ergodicity on small scales has been addressed in \cite{HR16, Han15}. In contrast with this case, because the quantitative estimate \eqref{e:QErate} depends only on the $L^2$ norm of the observable, we have no restriction on the size of the support of the observables and can choose observables with support on even smaller scales. In this case however, we can only deduce that there is no concentration of the mass of eigenfunctions at these very small scales, the rate is not strong enough to give equidistribution.

The quantitative statements of Theorem \ref{thm:quantitative} show in particular that it would be possible to weaken the condition of uniformity of the spectral gap and shortest geodesic. However, the conditions we require are typical for hyperbolic surfaces, in the sense that they are satisfied almost surely in the large scale limit for a natural model of random surfaces \cite{BM04}. The conditions are also satisfied for congruence coverings of compact arithmetic surfaces, where the level goes to infinity, making it possible to deduce immediately an equidistribution result for eigenfunctions in the level aspect in this setting (see for example \cite{Ber16} for an introduction to the arithmetic setting, and \cite[Theorem 5.2]{ABBGNRSPreprint} for a proof of the Benjamini-Schramm convergence in this case). In general, any sequence of cocompact lattices $(\Gamma_n)_{n\in\N}$ such that for all $n\in\N$, $\Gamma_{n+1} \subset \Gamma_n$, $\Gamma_n$ is normal in $\Gamma_0$, and $\bigcap_{n\in\N} \Gamma_n = \{ \id \}$, gives a sequence of surfaces $X_n = \Gamma_n \backslash \H$ satisfying conditions \eqref{condBS} and \eqref{condlmin} in  Theorem \ref{thm:sequences} (The sequence $X_n$ is called a \emph{tower of coverings} in this case and the radius of injectivity tends to $+\infty$ \cite[Theorem 2.1]{DGW78}).
Note that there is a relation between the spectral gap condition \eqref{condgap} and closed curves on the surface via the isoperimetric constant. The Cheeger isoperimetric constant is defined for a surface $X$ as
$$ h(X) = \inf_C \frac{\text{length}( C) }{\min (\text{area}(A), \text{area}(B))},$$
where $C$ runs over all closed curves on $X$ which divide $X$ into two pieces $A$ and $B$.
It is known that the Cheeger constant is essentially equivalent to the spectral gap in the sense that for a sequence $X_n$ of hyperbolic surfaces, $h(X_n) \to 0$ if and only if $\lambda_1^{(n)} \to 0$ when $n \to +\infty$ (See \cite{Bus82}). However even in the large injectivity radius case, $h(X_n)$ (and hence the spectral gap) need not be bounded from below. It is possible to construct towers of coverings $X_n$ such that $\lambda_1^{(n)} \to 0$ (See for example \cite{Bro86, AG12}).

\subsection{Further discussions and generalisations}

Theorem \ref{thm:sequences} is a Quantum Ergodicity type theorem, in the sense that we average over a spectral interval. 
By analogy with the large eigenvalue case, it is natural to ask if a stronger property holds.
\begin{question}\label{que:bs}
Do sequences of surfaces converging in the sense of Benjamini-Schramm satisfy a QUE property for fixed eigenvalues intervals? That is, if $I \subset (1/4,+\infty)$ is a compact interval, then under the assumptions of Theorem \ref{thm:sequences}, and reasonable regularity assumptions on the test functions $a_n$, do we have 
$$\max_{\lambda_j \in I}\left| \langle \psi_j^{(n)}, a_n \, \psi_j^{(n)} \rangle - \dashint a_n \, d\Vol_{X_n} \right| \To 0$$
as $(X_n)$ Benjamini-Schramm converges to $\H$? 
\end{question}
As far as we know this problem is open also in a weaker form where instead of taking the maximum over an interval we choose any sequence $(\psi_{j_n}^{(n)})$ of eigenfunctions and ask if 
$$\left| \langle \psi_{j_n}^{(n)}, a_n \, \psi_{j_n}^{(n)} \rangle - \dashint a_n \, d\Vol_{X_n} \right| \To 0.$$
Let us just mention again that an analog of this question for holomorphic forms has been addressed in \cite{Nel11} and \cite{NPS14}.

\medskip

The methods of the proof of Theorem \ref{thm:sequences} and Theorem \ref{thm:quantitative} are relatively elementary. They require only some kind of wave propagation and mixing properties of the dynamics. In this sense they have a potential to be generalised to a large range of different settings, as is already apparent from their use on graphs. We adapted the techniques to the specific case of hyperbolic surfaces. In particular we replace the wave propagator with renormalised averaging operators on discs of growing radius. This allows us to use an ergodic theorem of Nevo that becomes particularly natural in this context (see Section \ref{sec:exponentialdecay}). Two extensions of the result seem attainable with similar methods: the case of Maass forms on congruence coverings of the modular surface, and general observables in phase space. 
A more challenging problem would be to adapt the methods to the case of variable curvature or more general homogeneous spaces, in particular higher dimension. In this case it is unclear how the choice of propagator and the estimate on its spectral action would generalise.

We believe our result shows how the study of the discrete graph model can give new insights on the continuous case. There are also open problems on graphs that might benefit from the understanding of the manifold case, in particular the connection between non-regularity of the graphs and variable curvature.

If the history of QUE is any indication, this result is just the beginning of a new approach to the understanding of eigenfunctions of the Laplacian. The differences with the usual setting are yet to be explored, and we hope our methods can provide in particular new insights to the large eigenvalue theory. One of the striking differences is the minimal regularity required for the observables: is it a characteristic of the large scale setting or is it a strength of our method that could be transported to the large eigenvalue case? In the other direction, is it possible to find a proof of our theorem using the usual microlocal analysis approach? The large eigenvalue version of the Quantum Ergodicity theorem only requires ergodicity of the geodesic flow, but we use in a crucial way some mixing properties of the dynamics and a sufficient decay of correlations.

Our result is just a first step. It is difficult to predict the difficulty of the QUE problem in the large scale setting compared to the large eigenvalue one. We hope that this question will bring new perspectives and ideas to the field.

\bigskip
\textit{Organisation of the paper.} The paper is organised as follows. Section \ref{sec:prelim} contains some preliminaries on harmonic analysis on hyperbolic surfaces. In Section \ref{sec:strategy} we give the general idea of the proof of the main theorems. Section \ref{sec:wavemainthm} introduces a wave propagation operator and gives the proof of Theorems \ref{thm:sequences} and \ref{thm:quantitative} assuming two fundamental estimates: a Hilbert-Schmidt norm estimate and an estimate on the spectral action of the propagator. To prove the first estimate we introduce a general Hilbert-Schmidt norm estimate adapted to the case of Benjamini-Schramm convergence in Section \ref{sec:HSBS}. Then we invoke in Section \ref{sec:exponentialdecay} the main tool from ergodic theory we will be using, that is a mean ergodic theorem in $L^2$ for averaging operators. The proof of the Hilbert-Schmidt norm estimate is then carried out in \ref{sec:boundHSnorm}. The spectral action estimate is done in Section \ref{sec:spectralPt}. Finally Section \ref{sec:spectralBS} introduces an asymptotic estimate of the number of eigenvalues in a fixed bounded interval which can be seen as an analog of the Weyl law for the Benjamini-Schramm convergence setting.

\section{Elements of harmonic analysis and hyperbolic geometry} \label{sec:prelim}

In this section, we give some definitions and introduce elements of harmonic analysis on hyperbolic surfaces that we will use in the proof. For more background on the geometry and spectral theory of hyperbolic surfaces we refer to the books \cite{Bus10, Iwa02, Ber16}.

\subsection{Hyperbolic surfaces}
The hyperbolic plane is identified with the upper-half plane
$$\H = \{ z = x+iy \in \C \, | \, y > 0 \}, $$
equipped with the hyperbolic Riemannian metric
$$ ds^2 = \frac{dx^2 + dy^2}{y^2}. $$
We will denote by $d(z,z')$ the distance between two points $z,z' \in \H$.
The hyperbolic volume is given by
$$ d\mu(z) = \frac{dx \, dy}{y^2}. $$

The group of isometries of $\H$ is identified with $\PSL(2,\R)$, the group of real $2 \times 2$ matrices of determinant $1$ modulo $\pm \id$, acting by M\"obius transformations
$$ \left( \begin{pmatrix} a & b \\ c & d \end{pmatrix} \in \PSL(2,\R), z \in \H\right) \mapsto \begin{pmatrix} a & b \\ c & d \end{pmatrix} \cdot z = \frac{az + b}{cz + d}.$$

\bigskip

A \emph{hyperbolic surface} can be seen as a quotient $X = \Gamma \backslash \H$ of $\H$ by a discrete subgroup $\Gamma \subset \PSL_2(\R)$.
We denote by $D$ a \emph{fundamental domain} associated to $\Gamma$. If we fix $z_0 \in \H$, an example of a fundamental domain is given by the set
$$ D = \{ z \in \H \, | \, d(z_0, z) < d(z_0, \gamma z) \text{ for any } \gamma \in \Gamma - \{\pm \id \} \}. $$
The \emph{injectivity radius} on the surface $X = \Gamma \backslash \H$ at a point $z$ is given by
$$\InjRad_{X} (z) = \frac12 \min\{d(z, \gamma z) \: | \: \gamma \in \Gamma-\{\id\} \}.$$
Thus  $\InjRad_X(z)$ gives the largest $R > 0$ such that $B_X(z,R)$ is isometric to a ball of radius $R$ in the hyperbolic plane.
Let $g \in \PSL(2,\R)$, we define the translation operator $T_g$, such that for any function $f$ on $\H$
$$ T_g f(z) = f(g^{-1}\cdot z).$$
We will generally see a function $f$ on a hyperbolic surface $X = \Gamma \backslash \H$ as a $\Gamma$-invariant function $f : \H \to \C$,
$$ T_\gamma f(z) = f(\gamma^{-1} z) = f(z) \quad \text{for all } \gamma \in \Gamma.$$
The integral of the function on the surface is then equal to the integral of the invariant function over any fundamental domain
$$ \int_{D} f(z) \, d\mu(z).$$

\subsection{Geodesic flow}
The tangent bundle of $\H$ can be identified with $\H \times \C$. The hyperbolic metric gives the following inner product for two tangent vectors $(z,re^{i\theta})$ and $(z,r'e^{i\theta'})$ on the tangent plane $T_z\H$
$$\langle r e^{i\theta}, r' e^{i\theta'} \rangle_z  = \frac{r \, r'}{\Im(z)^2} \cos(\theta' - \theta). $$
As a consequence, the map
$$ (z,\theta) \in \H \times \S^1 \mapsto (z, \Im(z) \, e^{i\theta}) \in \H \times \C, $$
where $\S^1 = \R/2\pi\Z$, identifies $\H \times \S^1$ with the unit tangent bundle.

The group $\PSL(2,\R)$ acts on the tangent bundle via the differential of its action on $\H$. It is well known (see for example \cite{Kat92}) that this action induces a homeomorphism between 
$\PSL(2,\R)$ and the unit tangent bundle of $\H$, such that the action of $\PSL(2,\R)$ on itself by left multiplication corresponds to the action of $\PSL(2,\R)$ on the unit tangent bundle. 

We denote by $\phi_t : \H \times \S^1 \to \H \times \S^1$ the geodesic flow associated to $\H$. It is invariant with respect to the Liouville measure $d\mu \, d\theta$, where $d\theta$ is the Lebesgue measure on $\S^1$.
Via the identification $\H \times \S^1 \sim \PSL(2,\R)$, the geodesic flow is equal to the multiplication on the right by the diagonal subgroup
$$\phi_t(g) =  g \begin{pmatrix} e^{t/2} & 0 \\ 0 & e^{-t/2} \end{pmatrix}, \quad g \in G, t \in \R.$$

For a hyperbolic surface $\Gamma \backslash \H$, the unit tangent bundle is identified with $\Gamma \backslash \PSL(2,\R)$, and via this identification the geodesic flow will be given simply by
$$\phi_t(\Gamma g) =  \Gamma g \begin{pmatrix} e^{t/2} & 0 \\ 0 & e^{-t/2} \end{pmatrix}.$$

\subsection{Polar coordinates}
We will make use of the polar coordinates on $\H$. Let $z_0 \in \H$ be an arbitrary point. For any point $z\in\H$ different from $z_0$, there is a unique geodesic of length $r$ going from $z_0$ to $z$. Using the geodesic flow, it means that there is a unique $\theta \in \S^1$ and $r \in (0,\infty)$ such that 
$z$ is the projection of  $\phi_r(z_0,\theta)$ on the first coordinate. The change of variable
$ z \mapsto (r,\theta) $
is called \emph{polar coordinates}. The induced metric is
$$ ds^2 = dr^2 + \sinh^2 r \, d\theta^2, $$
and the hyperbolic volume in these coordinates is given by
$$ d\mu(r,\theta) = \sinh r \, dr \, d\theta. $$

\subsection{Invariant integral operators and Selberg transform}
In the coordinates $z = x+iy$, the Laplacian $\Delta$ on $\H$ is the differential operator
$$ \Delta = -y^2 \left(\frac{\partial^2}{\partial x^2} + \frac{\partial^2}{\partial y^2} \right). $$
A fundamental property of the Laplacian is that it commutes with isometries. 
We have for any $g\in\PSL(2,\R)$,
$$ T_g \Delta = \Delta T_g.$$
The Laplacian can therefore be seen as a differential operator on any hyperbolic surface $\Gamma \backslash \H$.

We say that a bounded measurable kernel $K : \H \times \H \to \C$ is \textit{invariant} under the diagonal action of $\Gamma$ if for any $\gamma \in \Gamma$ we have 
$$K(\gamma \cdot z, \gamma \cdot w) = K(z,w), \quad (z,w) \in \H \times \H.$$
Assume for simplicity that $K(z,w) = 0$ whenever $d(z,w) > C$ for some constant $C>0$.
For any $\Gamma$-invariant function $f$, such a kernel defines an integral operator $A$ on the surface $X$ by the formula
$$Af(z) = \int_{\H} K(z,w) f(w) \, d\mu(w) = \int_{D} \sum_{\gamma\in\Gamma} K(z,\gamma w) f(w) \, d\mu(w), \quad z \in D.$$
The function $k : \H \times \H \to \C$ given by
$$ k(z,w) = \sum_{\gamma\in\Gamma} K(z,\gamma w)$$ 
is such that $k(\gamma z, \gamma' w) = k(z,w)$ for any $\gamma,\gamma' \in \Gamma$, which defines the kernel on the surface.

A special case of invariant kernels is given by radial kernels. Let $k : [0,+\infty) \to \C $ be a bounded measurable compactly supported function, then
$$K(z,w) = k(d(z,w)), \quad (z,w) \in \H \times \H$$ 
is an invariant kernel.

For  $k : [0,+\infty) \to \C $, the \textit{Selberg transform} $\cS (k)$ of $k$ is obtained as the Fourier transform
$$ \mathcal S(k)(s) = \int_{-\infty}^{+\infty} e^{isu} g(u) \, du$$
of the Abel transform
$$g(u) = \sqrt{2} \int_{|u|}^{+\infty} \frac{k(\rho) \sinh \rho}{\sqrt{\cosh \rho - \cosh u}} \, d\rho. $$
For a function $h : \R \to \C$, the Selberg transform is inverted using the inverse Fourier transform
$$ g(u) = \frac1{2\pi}  \int_{-\infty}^{+\infty} e^{-isu} h(s) \, ds $$
 and the formula
$$k(\rho) = -\frac1{\sqrt{2}\pi} \int_\rho^{+\infty} \frac{g'(u)}{\sqrt{\cosh u - \cosh \rho}} \, du.$$

Eigenfunctions of the Laplacian are eigenfunctions of all operators of convolution by a radial kernel and the eigenvalues are given precisely by the Selberg transform.

\begin{prop}[\cite{Iwa02} Theorem 1.14]\label{t:Stransform}
Let $X = \Gamma \backslash \H$ be a hyperbolic surface. Let $k : [0,+\infty) \to \C$ be a smooth function with compact support. If $\psi_\lambda$ is an eigenfunction of the Laplacian on $X$ of eigenvalue $\lambda$, then it is an eigenfunction of the radial integral operator $A$ associated to $k$. That is,
$$A \psi_\lambda(z) =  \int k(d(z,w)) \psi_\lambda(w) d\mu(w) = h(s_\lambda) \psi_\lambda(z), $$
where the eigenvalue $h(s_\lambda)$ is given by the Selberg transform of the kernel $k$:
$$ h(s_\lambda) = \mathcal S(k) (s_\lambda),  $$
and $s_\lambda \in \C$ is defined by the equation $\lambda = \frac14 + s_\lambda^2$.
\end{prop}

Note that this statement can be generalised to the case of $k: [0,+\infty) \to \C$ measurable bounded and compactly supported by approximation and dominated convergence.

\section{Strategy of the proof}\label{sec:strategy}
Proving our theorem is equivalent to estimating
$$ \sum_{\lambda_j \in I} |\langle \psi_j, a \, \psi_j \rangle|^2$$
for an observable $a \in L^\infty$, such that $\int a \,d\mu = 0$.
\begin{enumerate}
\item For this purpose we introduce in Section \ref{sec:wavemainthm} a wave propagation operator $P_t$ that behaves to some extent like a quantum evolution operator (preserving the quantum average $\langle \psi_j, a \, \psi_j \rangle$ of an observable $a$). Although not unitary, we choose it such that there exists a constant $C_I >0$ depending only on the choice of the spectral interval $I$, so that
$$  \sum_{\lambda_j \in I} |\langle \psi_j, a \, \psi_j \rangle|^2 \leq C_I \sum_{\lambda_j \in I} \Big|\Big\langle \psi_j, \frac1T \int_0^T P_t \, a \, P_t \, dt \, \psi_j \Big\rangle\Big|^2.$$
 Studying the action of $P_t$ on eigenfunctions and proving this inequality is the object of Section \ref{sec:spectralPt}. Note that we also choose $P_t$ to have characteristics of a wave propagator, namely finite speed of propagation. This is because we will work locally, at a scale where the surfaces look like the hyperbolic plane around most of the points, and we need to prevent the propagation from going beyond this scale.
  
\item We then use the fact that the diagonal elements of a sufficiently regular operator in the orthonormal basis given by the eigenfunctions is bounded by the Hilbert-Schmidt norm of this operator:
$$ \sum_{\lambda_j \in I} \Big|\Big\langle \psi_j, \frac1T \int_0^T P_t \, a \, P_t \, dt \, \psi_j \Big\rangle\Big|^2 
\leq \left\| \frac1T \int_0^T P_t \, a \, P_t \, dt\right\|_{\HS}^2.$$
We give a way to estimate the Hilbert-Schmidt norm on a surface using the kernel of the operator on the plane in Section \ref{sec:HSBS}. 
For this we divide between the points with a small radius of injectivity, that are controlled by the Benjamini-Schramm convergence, and points with a radius of injectivity greater than $T$, the maximal distance of propagation.

\item We then study the kernel of the operator $P_t \, a \, P_t$. The kernel of this operator consists of averages of the observable $a$ over intersection of balls of radius $t$. To study these averages we introduce in Section \ref{sec:exponentialdecay} a mean ergodic theorem of Nevo for averages over sets of increasing volume. Using the decay estimate given by this theorem  and depending on the spectral gap, we would find that the norm of the operator $P_t \, a \, P_t$ can be bounded by a quantity independent of $t$:
$$ \|P_t \, a \, P_t\|_{\HS} \lesssim \|a\|_2. $$
The idea is then to show that we have an orthogonality property of the type
$$ \left\| \int_0^T P_t \, a \, P_t \, dt \right\|_{\HS}^2 = \int_0^T \|P_t \, a \, P_t\|_{\HS}^2 \, dt$$
such that
\begin{equation}\label{e:HSidea}
 \left\| \frac1T \int_0^T P_t \, a \, P_t \, dt \right\|_{\HS}^2 = \frac1{T^2} \int_0^T \|P_t \, a \, P_t\|_{\HS}^2 \, dt \leq \frac{\|a\|_2^2}{T}
 \end{equation}
whenever $T$ is less than the radius of injectivity. Assuming for simplicity that the radius of injectivity goes to infinity, we obtain the type of estimate that we want. In reality, the orthogonality property is not obtained exactly in terms of the variable $t$, and appears in the form of a change of variable lemma. The computation is carried out in Section~\ref{sec:boundHSnorm}.

\item In \eqref{e:HSidea}, the $L^2$ norm of the observable $a$ can be of the order of the area of the surface. We thus need to normalise our expression by dividing by this area. The link between the area and the number of eigenvalues in the interval $I$ is provided by an analog of Weyl's law, that we prove in Section \ref{sec:spectralBS}.
\end{enumerate}

\section{Wave propagation and reduction of the proof of Theorem \ref{thm:quantitative}} \label{sec:wavemainthm}

Let $X = \Gamma \backslash \H$ be a hyperbolic surface and $D \subset \H$ a fundamental domain.  We define the operator
$$P_t u(z) := \frac{1}{\sqrt{\cosh t}} \int_{B(z,t)} u(w) \, d\mu(w), \quad z \in \H,$$
for any function $u : \H \to \R$ and $t \geq 0$.
The operator $P_t$ can be seen as a regularised version of the wave propagator, or a renormalised averaging operator over balls of radius $t$. It is self-adjoint in $L^2(\H)$ and it is an integral operator associated to the radial $\Gamma$-invariant kernel
$$K_t(z,z') := k_t(d(z,z')), \quad z,z' \in \H \times \H,$$
where
$$ k_t(\rho) := \frac{1}{\sqrt{\cosh t}} \mathbf{1}_{\{ \rho \leq t \}}, \quad \rho \geq 0.$$
Fix now a test function $a \in L^2(\H)$. We still denote by $a$ the operator of multiplication by the function $a$. The kernel of the conjugation $P_t \, a \, P_t$ is given by
\begin{equation}\label{e:kernelconj}
[P_t \, a \, P_t](z,z') = \frac{1}{\cosh t} \int\limits_{B(z,t) \cap B(z',t)} a(w) \, d\mu(w), \quad z,z' \in \H.
\end{equation}
Note in particular that $[P_t \, a \, P_t](z,z') = 0$ whenever $d(z,z') > 2t$.
We want to estimate the Hilbert-Schmidt norm of the time averages,
$$ \Big\| \frac{1}{T} \int_0^T P_t \,a \,P_t \, dt \Big\|_{\HS}^2.$$
Reducing the Quantum Ergodicity theorem to an estimate on the Hilbert-Schmidt norm of time averages appears in the proof of the large eigenvalue version of the theorem \cite{Zel96}. The way we estimate this average differs however from the large eigenvalue case. Another difference is that we use the $L^2$-normalised averaging operator on discs $P_t$ instead of the standard propagator of the wave equation.
We prove the following bound:
\begin{prop}\label{p:HSbound} For any bounded $a \in L^2(\H)$, and for any $T > 0$, we have
	$$ \Big\| \frac{1}{T} \int_0^T P_t a P_t \, dt \Big\|_{\HS}^2 \lesssim \frac{\|a\|_2^2}{T\rho(\beta)^2} + \frac{e^{4T}}{\ell_{\min}} \Vol(\{ z \in X : \InjRad_{X}(z) < T\}) \|a\|_\infty^2,$$
	where $\rho(\beta)$ is a constant depending only on the spectral gap $\beta = \lambda_1$ of the Laplacian on $X$.
\end{prop}
We also need to know how the propagator $P_t$ acts on eigenfunctions of the Laplacian. This is given in the the following proposition:
\begin{prop}\label{prop:spectral}
The Selberg transform of the kernel $k_t$ is given by the integral
$$ h_t(s) = \frac{2}{\sqrt{\cosh t}}\int_0^t \cos(su) \sqrt{\cosh t  - \cosh u} \, du.$$
Moreover, for any fixed compact interval $I \subset (1/4,\infty)$, there exist constants $C_I,T_I > 0$ such that for all $s \in \R$ with $\lambda = 1/4 + s^2 \in I$, and for all $T \geq T_I$ we have
$$ \frac1T \int_0^T  h_t(s) ^2 \, dt \geq C_I.  $$
\end{prop}
From these propositions we can prove Theorem \ref{thm:quantitative}:
\begin{proof}[Proof of Theorem \ref{thm:quantitative}]
	Assume that $\int a \, d\Vol_X = 0$. Since $P_t$ is an operator of convolution with a radial kernel $k_t$, we have by Proposition \ref{t:Stransform} that the action of the operator $P_t$ on an eigenfunction $\psi_j$ with eigenvalue $\lambda_j$ is given by
	$$ P_t \psi_j = h_t(s_j) \psi_j,$$
	where $h_t$ is the Selberg transform of the kernel $k_t$ and $\lambda_j = 1/4 + s_j^2$. We deduce from this and the symmetry of $P_t$ that
	\begin{align*}
		\sum_{j : \lambda_j \in I} \Big|\Big\langle \psi_j, \Big(\frac1T \int_0^T P_t \, a \, P_t \, dt\Big) \, \psi_j \Big\rangle \Big|^2 
		&= \sum_{j : \lambda_j \in I} \left| \frac1T  \int_0^T h_t(s_j)^2 \, dt \right|^2 \, | \langle \psi_j,  a \psi_j \rangle |^2 .
	\end{align*}
	We know by Proposition \ref{prop:spectral} that there is a constant $C_I$ depending only on the interval $I$ such that
	$$\inf_{j : \lambda_j \in I} \left| \frac1T  \int_0^T h_t(s_j)^2 \, dt \right|^2 \geq C_I.$$
	We thus have by Proposition \ref{p:HSbound} that
	\begin{align*}
		\sum_{\lambda_j \in I} | \langle \psi_j,  a \psi_j \rangle |^2 
		&\leq \frac1{C_I} \sum_{\lambda_j \in I} \Big| \Big\langle \psi_j, \Big(\frac1T \int_0^T P_t \, a \, P_t \, dt\Big) \, \psi_j \Big\rangle \Big|^2 \\
		&\lesssim_I \left\| \frac1T \int_0^T P_t \, a \, P_t \, dt \right\|_{\HS}^2 \\
		&\lesssim_I \frac{\|a\|_2^2}{T\rho(\beta)} + \frac{e^{4R}}{\ell_{\min}} \Vol(\{ z \in X : \InjRad_{X}(z) < R\}) \|a\|_\infty^2
	\end{align*}
	When we divide by the volume we obtain 
	$$ \frac1{\Vol(X)} \sum_{s_j \in I} | \langle \psi_j,  a \psi_j \rangle |^2 
	\lesssim_I \frac{\|a\|_\infty^2}{T\rho(\beta)} + \frac{e^{4T}}{\ell_{\min}} \frac{\Vol(\{ z \in X : \InjRad_{X}(z) < T\})}{\Vol(X)} \|a\|_\infty^2.$$
	Then Lemma \ref{lma:boundingeigenvalues} on the relationship between the number $N(X,I)$ of eigenvalues $\lambda_j \in I$ and the volume $\Vol(X)$ yields the claim of Theorem \ref{thm:sequences}
\end{proof}

To complete the proof of the main theorems, we need to prove Proposition \ref{p:HSbound} and Proposition \ref{prop:spectral}. 
In order to prove the first proposition, we start with two preliminary sections. We first give a way to estimate the Hilbert-Schmidt norm of the operators we are considering that depends on the radius of injectivity. Then we introduce a mean ergodic theorem in $L^2$ for averaging operators over sets of increasing volume.

\section{Hilbert-Schmidt norm and injectivity radius}\label{sec:HSBS}

The \textit{Hilbert-Schmidt norm} of a bounded operator $A$ on $X$ is given by the trace of $A^* A$
$$\|A\|_{\HS}^2 = \mathrm{Tr}(A^* A),$$
where $A^*$ is the adjoint ot $A$. If $K$ is the kernel of $A$, then
$$ \|A\|_{\HS}^2 = \iint_{D\times D} |K(z,z')|^2 d\mu(z) d\mu(z').$$
The following lemma tells us how to compute the Hilbert-Schmidt norm of integral operators with $\Gamma$-invariant kernels on the surface $X = \Gamma \backslash \H$ in terms of the injectivity radius, connecting it to the Benjamini-Schramm convergence.

\begin{lemma}\label{l:normsurface}
Let $K : \H \times \H \to \R$ be invariant under the diagonal action of $\Gamma$. The integral operator $A$ on $X$ defined by the kernel has the following Hilbert-Schmidt norm:
$$\| A \|_{\HS}^2 = \int_D \int_D | \sum_{\gamma \in \Gamma} K(z, \gamma \cdot w) |^2 d\mu(z) d\mu(w). $$
If $K(z,w) = 0$ whenever $d(z,w) \geq R$, we have
\begin{equation}\label{e:HSnorm}
\begin{aligned}
\| A \|_{\HS}^2 &\leq \int_D \int_\H  |K(z, w) |^2 d\mu(z) d\mu(w) \\
&\quad+ \frac{e^{2R}}{\ell_{\min}} \Vol\{ z \in X : \InjRad_{X}(z) < R\} \sup_{(z,w) \in D\times\H} |K(z,w)|^2,
\end{aligned}
\end{equation}
where $\ell_{\min}$ denotes the length of the shortest closed geodesic on $\Gamma\backslash\H$.
\end{lemma}

\begin{proof}
We denote by $D(R)$ the point in the fundamental domain $D$ with radius of injectivity greater than $R$:
$$D(R) = \{ z \in D  :  \InjRad_X(z) \geq R \},$$
and by $D(R)^c$ the complement of this set in $D$.
We then split the integral into two parts, and use that in the first part, the sum over $\Gamma$ is reduced to one term.
\begin{align*}
	\| A \|_{\HS}^2 = 
	\int_{D(R)} \int_D \sum_{\gamma \in \Gamma} | K(z, \gamma \cdot w) |^2 d\mu(z) d\mu(w) + \int_{D(R)^c} \int_D | \sum_{\gamma \in \Gamma} K(z, \gamma \cdot w) |^2 d\mu(z) d\mu(w)
\end{align*}
For the second sum, we use that it contains at most $e^R/\ell_{\min}$ terms. Indeed, we know that the volume of the ball of radius $R$ is equal to $\cosh R - 1$ and the minimal distance between two lattice points is $\ell_{\min}$, the number of lattice points is thus bounded by $(\cosh(R+\ell_{\min}) - 1)/(\cosh\ell_{\min} -1)$, which is bounded by $e^R/\ell_{\min}$. 

We get using Cauchy-Schwarz inequality
$$\| A \|_{\HS}^2 \leq 
	\int_{D} \int_\H | K(z, w) |^2 d\mu(z) d\mu(w) + \frac{e^R}{\ell_{\min}} \int_{D(R)^c} \int_D \sum_{\gamma \in \Gamma} |K(z, \gamma \cdot w) |^2 d\mu(z) d\mu(w).$$
The second term on the right-hand side is bounded by 
\begin{align*}
	&\frac{e^R}{\ell_{\min}}\sup_{(z,w) \in D\times\H} |K(z,w)|^2  \int_{D(R)^c} \Vol(B_\H(z,R)) \, d\mu(z)\\
	&\quad \leq \frac{e^{2R}}{\ell_{\min}} \Vol\{ z \in X : \InjRad_{X}(z) < R\} \sup_{(z,w) \in D\times\H} |K(z,w)|^2
\end{align*}
\end{proof}

\section{Mean ergodic theorem for averaging operators}\label{sec:exponentialdecay}

Our main tool from ergodic theory is given by an equidistribution result for averaging operators over general sets. Let us first introduce the problem in its general setting.

Let $(Y,\nu)$ be a probability space, and $G$ a group equipped with left-invariant Haar measure $dg$, and a measure-preserving action on $Y$. The action of $G$ on $Y$ defines a representation on $L^2(Y)$ by
$\pi_Y(g) f(x) = f(g^{-1} x)$ for any $f\in L^2(Y)$. 
For a collection of measurable sets $F_t \subset G$ we are interested in the behaviour of the averaging operators defined by
$$\pi_Y(F_t)f(x) = \frac{1}{|F_t|} \int_{F_t} \, f(g^{-1} x) \, dg, \quad f \in L^2(Y), \quad x \in Y,$$
as $t \to +\infty$.

Where by abuse of notation we denote by $\pi_Y$ both the representation and the averaging operator.
An extensive study of this question has been published by Gorodnik and Nevo \cite{GN10} (see also the survey \cite{GN15}).
We will use in particular a theorem of Nevo \cite[Theorem 4.1]{GN15}. To state the theorem we first need the notion of integrability exponent.
Define
$$L^2_0(Y) := \Big\{ f\in L^2(Y) : \int_Y f \, d\nu = 0 \Big\}.$$
and the matrix coefficient$$C_{f,h}(g) := \langle \pi_Y(g) f, h \rangle,$$
for any $f,h \in L^2_0(Y)$ and $g \in G$. The \emph{integrability exponent} of ${\pi_Y}_{|L^2_0(Y)}$ is given by
$$q_0 := \inf \{ q > 0 : C_{f,h} \in L^q(G) \text{ for all  $f,h$ in a dense subset of }L^2_0(Y) \}.$$
We say that the action of $G$ on $Y$ has a \textit{spectral gap} if $q_0 < \infty$ (see \cite[Section 3.2]{GN15}). 

\begin{theorem}[Nevo {\cite{Nev98}}]\label{thm:nevo}
If $G$ is a connected simple Lie group and the measure-preserving action on the probability space $(Y,\nu)$ has a spectral gap, then there exist $C, \theta > 0$ such that for any family $F_t \subset G$, $t \geq 0$, of measurable sets of positive measure, we have 
$$\left\| \pi_Y(F_t) f - \int_Y f \, d\mu\right\|_{L^2(Y,\nu)} \leq C \, |F_t|^{-\theta} \, \|f\|_{L^2(Y,\nu)}$$
for any $f \in L^2(Y,\nu)$, where we denote by $|F_t|$ the measure of the set $F_t$. The constant $C$ depends only on $G$ and $\theta$ depends only on the integrability exponent.
\end{theorem}

For a family $F_t$ with volume increasing in $t$, Theorem \ref{thm:nevo} gives a rate at which the averages $\pi_Y(F_t) f$ converge to $\int_Y f \, d\mu$ in $L^2(Y)$, and this rate only depends on the integrability exponent and volume of $F_t$. Let us insist on the fact that no further assumption on $F_t$ is needed. 

We will apply this theorem in the next section to $G = \PSL(2,\R)$, $Y = \Gamma \backslash \H$ to bound the Hilbert-Schmidt norms of the time averages $\frac{1}{T} \int_0^T P_t a P_t \, dt$. The quantitative decay of the matrix coefficients for representations of $G$ is well known (see \cite[Chapter V. Proposition 3.1.5]{HT}) and for the representation ${\pi_Y}_{|L^2_0(Y)}$ depends only on the spectral gap of the Laplacian. The integrability exponent is therefore determined by the spectral gap of the Laplacian on $\Gamma \backslash \H$.

\section{Bounding the Hilbert-Schmidt norm} \label{sec:boundHSnorm}

To prove Proposition \ref{p:HSbound} using Lemma \ref{l:normsurface} applied to the kernel \eqref{e:kernelconj} of $P_t \, a \, P_t$, it is sufficient to bound the quantity
	\begin{align*}
		&\int_D \int_{\H} \Big| \frac{1}{T} \int_0^T [P_t a P_t](z,z') \, dt\Big|^2 \, d\mu(z') \, d\mu(z)  \\
		&\quad = \int_D \int_{\H} \Big| \frac{1}{T} \int_0^T  \cosh(t)^{-1} \int\limits_{B(z,t) \cap B(z',t)} a(w) \, d\mu(w) \, dt \Big|^2 \, d\mu(z) \, d\mu(z').
	\end{align*}
In order to use the $L^2$ estimate for ergodic averages given in Theorem \ref{thm:nevo}, we will rewrite the previous expression as the $L^2$ norm of an averaging operator acting on the function $a$. For this purpose we introduce a change of variable lemma that allows us to go from the double integral over the surface to an integral over the unit tangent bundle. 

For $R > 0$ write $B_R = \{ (z,z') \in D \times \H \, | \, d(z,z') < R \}$.  We define the mapping $ \Phi : B_R \to D \times \S^1 \times (0,R)$ by 
$$\Phi(z,z') = (m(z,z'), \theta(z,z'), d(z,z')),$$
where $m(z,z')$ is the middle of the geodesic segment $[z,z']$, $\theta(z,z')$ is the direction of a unit vector in $m(z,z')$ tangent to this segment, and $d(z,z')$ is the geodesic distance between $z$ and $z'$. This gives us the following change of variable. 

\begin{lemma}\label{l:middle integration}
For any $f : \H \times \S^1 \times [0,\infty) \to \C$ invariant under the action of $\Gamma$, i.e. such that
$$ \forall \gamma \in \Gamma, \quad f(\gamma \cdot (z, \theta) , r) = f(z, \theta, r),$$ 
we have
$$\iint_{B_R} f(m(z,z'), \theta(z,z'), d(z,z')) \, d\mu(z) \,d\mu(z') =  \int_0^R \sinh(r) \int_D \int_{\S^1} f(z,\theta,r) \, d\theta \, d \mu(z) \, dr.$$
\end{lemma}

\begin{figure}[ht!]
\includegraphics[scale=0.40]{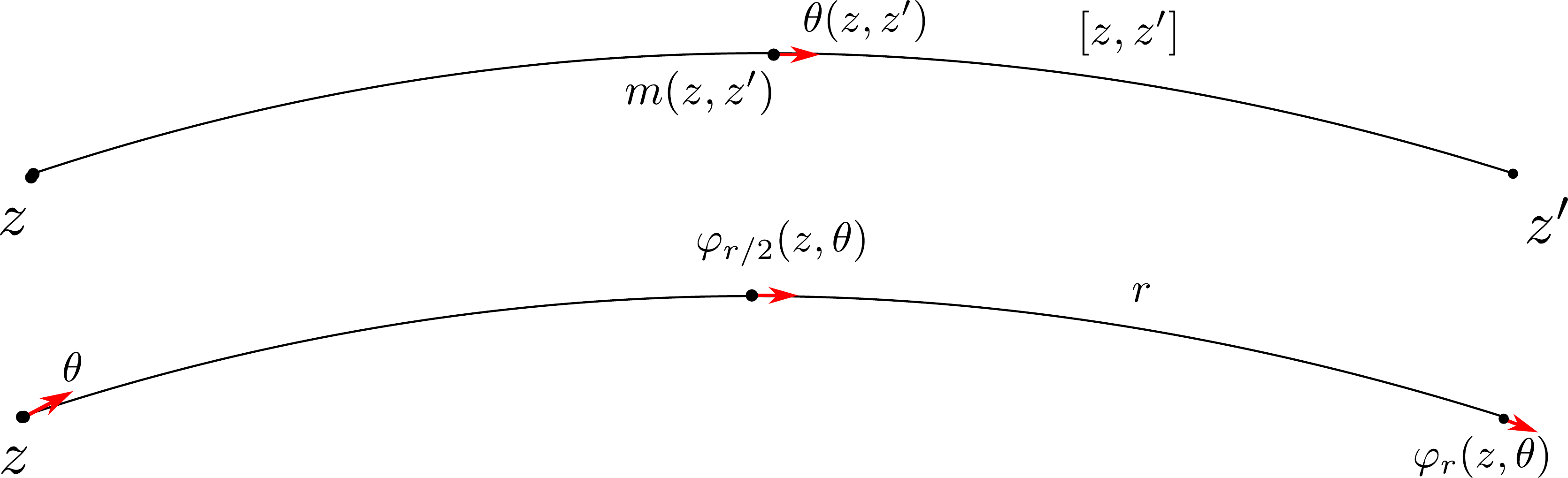}
\caption{Lemma \ref{l:middle integration} allows us to go from the double integration over middle points to an integral on the unitary tangent bundle.}
\label{fig:2}
\end{figure}

\begin{proof} 
We fix the variable $z \in D$, and we use for $z' \in \H$ the polar coordinates centred in $z$,
$$ z' \mapsto (\theta, r),$$ 
where $r = r(z,z') = d(z,z')$ and $\theta = \theta(z,z')$ is the direction of the unit vector at $m(z,z')$ tangent to the geodesic $(z,z')$, see Figure \ref{fig:2}. As we have $d\mu(z') = \sinh r \, dr d\theta$, this change of variable gives us
$$\iint_{B_R} f(m(z,z'), \theta(z,z'), d(z,z')) \, d\mu(z) d\mu(z') = \int_0^R \int_{\S^1} \int_D f(\phi_{r/2}(z,\theta),r) \, \sinh(r) \, d \mu(z) \, d\theta \, dr,$$
with $\phi_{r/2}$ the geodesic flow at time $r/2$.

Now we note that $d\mu(z) \, d\theta$ is the Liouville measure, which is invariant under the action of the geodesic flow $(z, \theta) \mapsto \phi_{r/2}(z,\theta)$. We thus have
$$\iint_{B_R} f(m(z,z'), \theta(z,z'), d(z,z')) \, d\mu(z) d\mu(z') = \int_0^R \int_{\S^1} \int_D f(z,\theta,r) \, \sinh(r) \, d \mu(z) \, d\theta \, dr.$$
\end{proof}

We are now ready to prove the proposition.

\begin{proof}[Proof of Proposition \ref{p:HSbound}]
We have 
	\begin{align*}
		&\int_D \int_{\H} \Big| \frac{1}{T} \int_0^T [P_t a P_t](z,z') \, dt\Big|^2 \, d\mu(z') \, d\mu(z)  \\
		&\quad = \int_D \int_{\H} \Big| \frac{1}{T} \int_0^T  \cosh(t)^{-1} \int\limits_{B(z,t) \cap B(z',t)} a(w) \, d\mu(w) \, dt \Big|^2 \, d\mu(z) \, d\mu(z')
	\end{align*}

	For every $r \in (0,2T)$, we define the family of sets $F_t(r) \subset \PSL(2,\R)$ such that for any point $(z,\theta) \in \H\times\S^1$, the set $F_t(r)^{-1} \cdot (z,\theta)$ is the lift to the unit tangent bundle of the intersection of two balls of radius $t$, where the two centres are given by the projection on the surface of $\phi_{-r/2} (z,\theta)$ and $\phi_{r/2} (z,\theta)$.	
	
			\begin{figure}[htbp]
\includegraphics[scale=0.27]{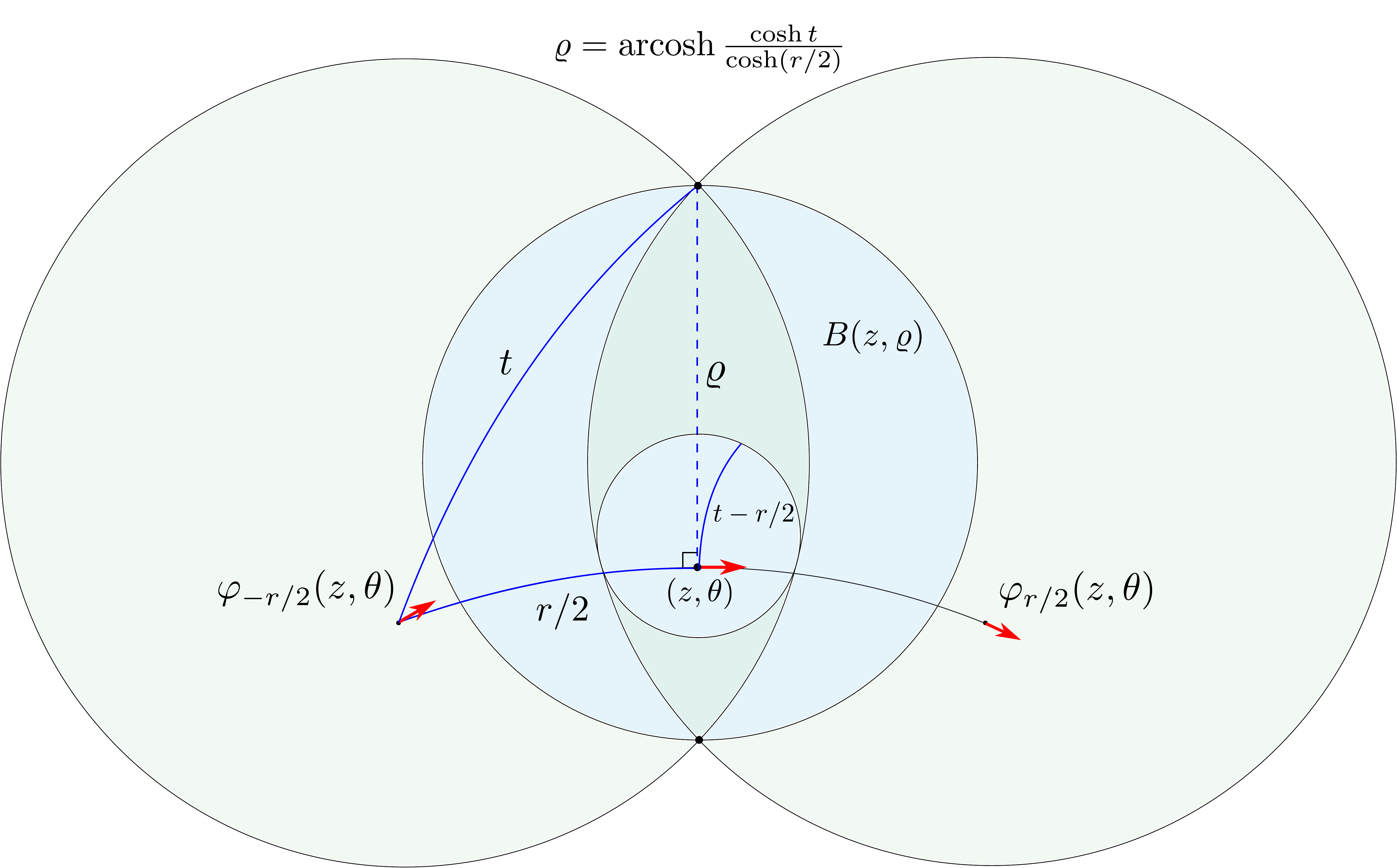}
\caption{The volume of the sets $F_t(r)$ used in the proof of Proposition \ref{p:HSbound} can be controlled by the volume of the balls $B(z,t-r/2)$ and $B(z,\rho)$, where $\cosh \rho = \frac{\cosh t}{\cosh(r/2)}$ by the hyperbolic version of Pythagoras' theorem. The volume of both of these balls is $O(e^{t-r/2})$.}
\label{fig:3}
\end{figure}
	
	We write for any function $f \in L^2(\Gamma\backslash\H)$
	$$ \pi(F_t(r)) f (z,\theta) = \frac1{|F_t(r)|} \int_{F_t(r)} f(g^{-1} (z,\theta)) \, dg, $$
	and we have, using the change of variable Lemma \ref{l:middle integration},
	\begin{align*}
		&\int_D \int_{\H} \Big| \frac{1}{T} \int_0^T [P_t a P_t](z,z') \, dt\Big|^2 \, d\mu(z') \, d\mu(z)  \\
		&\quad = \int_0^{2T} \sinh r \int_D \int_{\S^1} \left|\frac{1}{T} \int_{r/2}^T \cosh(t)^{-1} |F_t(r)| \, \pi(F_t(r)) a (z,\theta) \, dt\right|^2 \, d\mu(z) \, d\theta \,dr, \\
		&\quad \leq \int_0^{2T} \sinh r \left( \frac{1}{T} \int_{r/2}^T \cosh(t)^{-1} |F_t(r)|
		\left\| \pi(F_t(r)) a \right\|_{L^2(D\times\S^1)} \, dt\right)^2\,dr.
	\end{align*}
	Where we used Minkowski's integral inequality to obtain the last line.

	We then apply Theorem \ref{thm:nevo} with $G = \PSL(2,\R)$ and $Y = \Gamma \backslash \PSL(2,\R)$, identifying the latter to $D \times \S^1$. We obtain that there is a constant $\rho(\beta) > 0$ depending only on the spectral gap $\beta$ of the Laplacian such that the previous quantity is bounded up to a uniform constant by
	\begin{align*}
		\int_0^{2T} \sinh r \left( \frac{1}{T} \int_{r/2}^T \cosh(t)^{-1} |F_t(r)|^{1-\rho(\beta)}
		\left\| a \right\|_2 \, dt\right)^2\,dr.
	\end{align*}
	We then remark that the volume of the intersection of balls of radius $t$ with centres at a distance $r$ from each other is $O(e^{t-r/2})$, which gives us an estimate of $|F_t(r)|$.
Hence we have
	\begin{align*}
		&\int_D \int_{\H} \Big| \frac{1}{T} \int_0^T [P_t a P_t](z,z') \, dt\Big|^2 \, d\mu(z') \, d\mu(z)\\
		&\quad \lesssim \int_0^{2T} \sinh r \left( \frac{1}{T} \int_{r/2}^T e^{-r/2} e^{-\rho(\beta)(t-r/2)}
		\left\| a \right\|_2 \, dt\right)^2\,dr\\
		&\quad \lesssim \frac{1}{T^2}  \int_0^{2T} \frac{\|a\|_2^2}{\rho(\beta)^2} \,dr\\
		&\quad \lesssim \frac{\|a\|_2^2}{T \rho(\beta)^2}.
	\end{align*}	
\end{proof}

\section{Spectral action of the propagator} \label{sec:spectralPt}

In this section, we study the spectral action of the propagator $P_t$ and prove Proposition \ref{prop:spectral}. We first give the expression for the Selberg transform we will need:

\begin{lemma}
The Selberg transform of the kernel $$k_t(\rho) = \frac{1}{\sqrt{\cosh t}} \1_{\{\rho \leq t\}}$$ 
is given by the integral
\begin{equation}\label{e:selbergkernel}
 h_t(s) = 2\sqrt{2}\int_0^t \cos(su) \sqrt{1  - \frac{\cosh u}{\cosh t}} \, du.
\end{equation}
\end{lemma}

\begin{proof}
Applying the Selberg transform to the kernel $k_t$, we have
$$ h_t(s) = \int_{-\infty}^{+\infty} e^{isu} g(u) \, du $$
with
\begin{align*}
g(u) = \sqrt{\frac2{\cosh t}} \int_{|u|}^{t} \frac{ \sinh \rho}{\sqrt{\cosh \rho - \cosh u}} \, d\rho
=  \sqrt{\frac2{\cosh t}} \, \sqrt{\cosh t - \cosh u} \: \: \mathbf{1}_{(-t,t)}(u).
\end{align*}
Using the fact that $\cosh(-u) = \cosh(u)$ we obtain \eqref{e:selbergkernel}.
\end{proof}

Now to obtain Proposition \ref{prop:spectral}, we will need that for a fixed compact interval $I \subset (1/4,\infty)$, we want to find constants $C_I,T_I > 0$ such that for all $\lambda \in I$ with $\lambda = 1/4 + s^2$ and $T \geq T_I$ we have
\begin{align}\label{eq:spectralneededbound} \frac1T \int_0^T h_t(s) ^2 \, dt \geq C_I.  \end{align}
To simplify the notation we will denote also by $I$ the interval of the parameter $s$ and write $s\in I$.

\begin{figure}[ht!]
\includegraphics[scale=0.37]{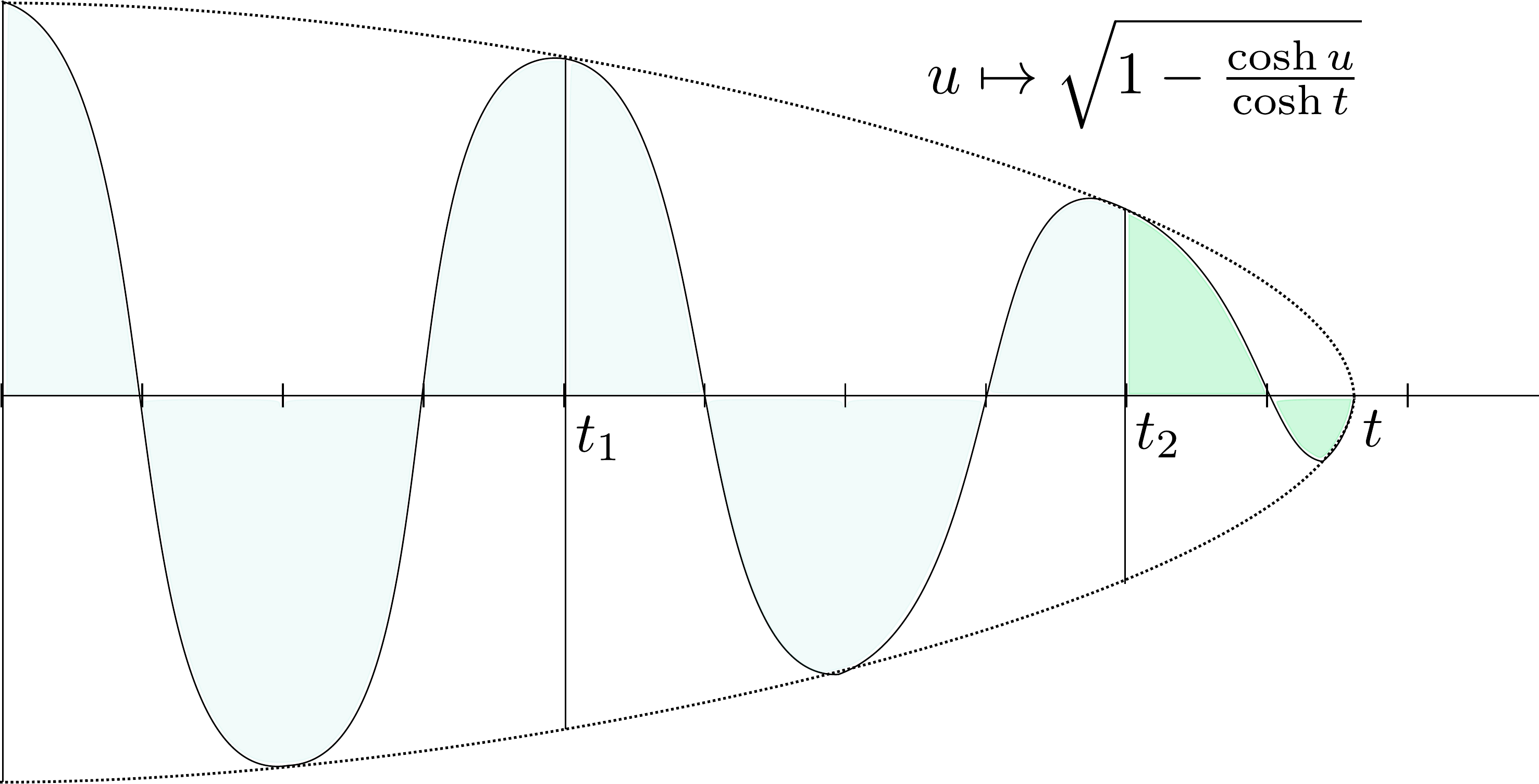}
\caption{Illustration of the Selberg transform $h_t(s)$ of the kernel $k_t(\rho)$ with $k = 2$. The constant $c(I)$ is constructed from the integral up to time $t_2$ (in Lemma \ref{prop:spectralperiodbound}) using the compactness of $I$. The remaining part from $t_2$ to $t$ may subtract mass from this, but we can control this uniformly over $k$ due to the uniform Lipschitz continuity of $t \mapsto h_t(s)$ in Lemma \ref{prop:spectrallipschitz}.}
\label{fig:4}
\end{figure}

The idea of the proof of \eqref{eq:spectralneededbound} is to show first that the function $t \mapsto h_t(s)$ satisfies a Lipschitz regularity condition (see Lemma \ref{prop:spectrallipschitz} below), and then use the fact that \textit{uniformly} over some increasing sequence $(t_k)_{k \in \N}$ such that $t_k \to \infty$, there is a given positive lower bound for $| h_t(s) |^2$, depending only on $s$ (see Lemma \ref{prop:spectralperiodbound} below). See Figure \ref{fig:4} for an illustration of the proof.

Let us first prove the  Lipschitz regularity of $t \mapsto h_t(s)$:

\begin{lemma}\label{prop:spectrallipschitz}
The function
$$t \mapsto h_t(s)$$
is uniformly Lipschitz on $(1,\infty)$ with constant independent of $s$.
\end{lemma}

\begin{proof}

Let us assume that $t > 1$. We are going to prove that the derivative of $t \mapsto h_t(s)$ is bounded uniformly in $s$. 
The derivative is given by
\begin{align*}
|\partial_t h_t(s)| 
&= \frac{\sqrt{2}}{\cosh(t)^{3/2}} \left| \int_0^t \cos(su) \frac{\sinh t \cosh u}{ \sqrt{\cosh t - \cosh u}} \, du \right| \\
&\leq \frac{\sqrt{2}}{\sqrt{\cosh t}} \int_0^t \frac{\cosh u}{ \sqrt{\cosh t - \cosh u}} \, du.
\end{align*}
Fix $1 < \delta < t$. We divide the integral into two parts: on the interval $(\delta, t)$ we have the inequality
$$ \cosh u \leq \coth \delta \sinh u = \frac{\cosh \delta}{\sinh \delta} \sinh u,$$
and
\begin{align*}
\left|\frac{ \coth \delta}{\sqrt{\cosh t}} \int_\delta^t \frac{\sinh u}{ \sqrt{\cosh t - \cosh u}}\right|
&= \frac{2 \coth \delta}{\sqrt{\cosh t}} \sqrt{\cosh t - \cosh \delta} \\
&\leq 2 \coth \delta.
\end{align*}
On the interval $(0,\delta)$ we have the bound, using that $t > 1$
\begin{align*}
\left|\frac{1}{\sqrt{\cosh t}} \int_0^\delta \frac{\cosh u}{ \sqrt{\cosh t - \cosh u}} \, du\right|
&\leq \frac{\delta}{\sqrt{\cosh t}} \frac{\cosh \delta}{ \sqrt{\cosh t - \cosh \delta}}\\
&\leq \delta \left( \sqrt{\frac{\cosh 1}{\cosh \delta} - 1}\right)^{-1/2}.
\end{align*}
We thus obtain a bound for $|\partial_t h_t(s)| $ that is uniform in $s$ and $t$ for any choice of $\delta$, and we deduce that the function $h_t(s)$ is Lipschitz in $t \in (1,+\infty)$ uniformly in $s$ by the mean value theorem.
\end{proof}

Now let us construct the desired sequence $(t_k)_{k \in \N}$, for which we have uniform bounds for $|h_t(s)|^2$ in $t$ and over $s \in I$. For $s > 0$ and $k \in \N$ write
$$t_k := \frac{2\pi k}{s}.$$

\begin{lemma}\label{prop:spectralperiodbound}
Given a bounded interval $I$, there exists a constant $c(I) > 0$ and $k_0(I) \in \N$ such that for any $k \geq k_0(I)$ and any $s \in I$ we have
$$h_{t_k}(s) < -2c(I).$$
\end{lemma}

\begin{proof}
First of all, for any $j = 1,\dots,k-1$ we have
\begin{align*}
\int_{t_{j-1}}^{t_j} \cos(su) \sqrt{1 - \frac{\cosh u}{\cosh t_k}} \, du & = \frac{1}{2s\cosh t_k} \int_{t_{j-1}}^{t_j} \sin(su) \frac{\sinh u}{\sqrt{1 - \frac{\cosh u}{\cosh t_k}}} \, du < 0
\end{align*}
since the map $u \mapsto (\sinh u)/\sqrt{1 - \frac{\cosh u}{\cosh t_k}}$ is continuous and monotonically increasing on $[0,t_k)$. Thus it is enough to find $c(I) > 0$ such that
\begin{align}\label{eq:wehavealready}\int_{t_{k-1}}^{t_k} \cos(su) \sqrt{1 - \frac{\cosh u}{\cosh t_k}} \, du < -c(I).\end{align}
The idea is to obtain a bound that is independent of $k$ and $s \in I$. Recall that $t_k = \frac{2\pi k}{s}$. Let us do a change of variable $v :=  u - t_{k-1}$ which maps $[t_{k-1},t_k]$ onto $[0,2\pi/s]$. Then as $\cos(s(v+t_{k-1})) = \cos(sv+2\pi(k-1)) = \cos(sv)$ for all $v \in [0,2\pi/s]$, we have
$$\int_{t_{k-1}}^{t_k} \cos(su) \sqrt{1 - \frac{\cosh u}{\cosh t_k}} \, du = \int_{0}^{\frac{2\pi}{s}} \cos(sv) \sqrt{1 - \frac{\cosh (v +t_{k-1})}{\cosh t_k}} \, dv.$$
For $v \in [0,2\pi/s]$ denote
$$f_k(v) := \sqrt{1 - \frac{\cosh (v +t_{k-1})}{\cosh t_k}}.$$
Then as $t_{k-1} = t_k - \frac{2\pi}{s}$, we have
\begin{align*}f_k(v) = \sqrt{1 - \frac{ e^{v+t_{k-1}} + e^{-(v+t_{k-1})}}{ e^{t_k} + e^{-t_k} }}  = \sqrt{1 - e^{v-\frac{2\pi}{s}} \cdot \frac{1 + e^{-(2v + 2t_k - \frac{4\pi}{s})}}{1+e^{-2t_k} }} .\end{align*}
Hence we have that the sequence of functions $(f_k)_{k \in \N}$ converges (as $k \to \infty$) pointwise to the limit function
$$f(v) := \sqrt{1 - e^{v-\frac{2\pi}{s}}}.$$
Thanks to the compactness of $I$, this convergence is uniform over $s \in I$ as we will see now. Write $I = [a,b]$, fix $s \in I$ and $v \in [0,2\pi/s]$. As $s \geq a$, $t_k = 2 \pi k/s$ and $v \geq 0$ we obtain
$$1 \leq \frac{1 + e^{-(2v + 2t_k - \frac{4\pi}{s})}}{1+e^{-2t_k} } \leq 1+e^{\frac{4(1-k)\pi}{a}}.$$
Thus $f_k(v) \leq f(v)$ and
$$f_k(v) \geq \sqrt{1 - e^{v-\frac{2\pi}{s}} \cdot (1+e^{\frac{4(1-k)\pi}{a}})}.$$
By the mean value theorem, we can find $\xi \in [1 - e^{v-\frac{2\pi}{s}} \cdot (1+e^{\frac{4(1-k)\pi}{a}}),1 - e^{v-\frac{2\pi}{s}}]$ satisfying
$$f(v) - \sqrt{1 - e^{v-\frac{2\pi}{s}} \cdot (1+e^{\frac{4(1-k)\pi}{a}})} = \frac{e^{v-\frac{2\pi}{s}}}{2\sqrt{\xi}} \cdot e^{\frac{4(1-k)\pi}{a}}.$$
Since $e^{v-\frac{2\pi}{s}} \leq 1$ we have by the choice of $\xi$ that
$$\frac{e^{v-\frac{2\pi}{s}}}{2\sqrt{\xi}} \cdot e^{\frac{4(1-k)\pi}{a}} \leq \frac{1}{2\sqrt{e^{\frac{4(1-k)\pi}{a}}}} \cdot e^{\frac{4(1-k)\pi}{a}} = \frac{1}{2} e^{\frac{2(1-k)\pi}{a}}.$$
Thus we have proved
$$\chi_{[0,2\pi/s]}(v)  |f_k(v) - f(v)|\leq \frac{1}{2} e^{\frac{2(1-k)\pi}{a}}$$
for all $s \in I$ and $v \geq 0$. Therefore
\begin{align*}& \sup_{s \in I} \Big|\int_{0}^{\frac{2\pi}{s}} \cos(sv) \sqrt{1 - \frac{\cosh (v +t_{k-1})}{\cosh t_k}} \, dv-\int_{0}^{\frac{2\pi}{s}} \cos(sv) \sqrt{1 - e^{v-\frac{2\pi}{s}}} \, dv\Big|\\
& \leq \sup_{s \in I} \int_{0}^{\frac{2\pi}{a}} \chi_{[0,2\pi/s]}(v) | f_k(v) - f(v)| \, dv\\
& \leq \frac{\pi}{a} e^{\frac{2(1-k)\pi}{a}}.
\end{align*}
Now for any $s > 0$ write
$$c(s) := -\frac{1}{2}\int_{0}^{\frac{2\pi}{s}} \cos(sv) \sqrt{1 - e^{v-\frac{2\pi}{s}}} \, dv.$$
Then $c(s) > 0$. Indeed, a change of variable $x := sv$ and integration by parts gives
$$\int_{0}^{\frac{2\pi}{s}} \cos(sv) \sqrt{1 - e^{v-\frac{2\pi}{s}}} \, dv = \frac{1}{s}\int_0^{2\pi} \cos(x) \sqrt{1-e^{\frac{x-2\pi}{s}}} \, dx = \frac{1}{s^2}\int_0^{2\pi} \sin(x)\frac{e^{\frac{x-2\pi}{s}}}{\sqrt{1-e^{\frac{x-2\pi}{s}}}}\, dx$$
so the positivity follows as $x \mapsto e^{\frac{x-2\pi}{s}}/\sqrt{1-e^{\frac{x-2\pi}{s}}}$ is strictly increasing on $[0,2\pi)$.

Hence, as $I$ is compact and $s \mapsto c(s)$ is continuous, we have that the infimum
$$c(I) := \inf_{s \in I} c(s) > 0.$$
Since $\frac{\pi}{a} e^{\frac{2(1-k)\pi}{a}} \to 0$ as $k \to \infty$ and this sequence only depends on $I$, the above computation yields $k_0(I) \in \N$ such that for all $k \geq k_0(I)$ and all $s \in I$ we have that
$$\Big|\int_{0}^{\frac{2\pi}{s}} \cos(sv) \sqrt{1 - \frac{\cosh (v +t_{k-1})}{\cosh t_k}} \, dv + 2c(s)\Big| \leq c(I).$$
Since $c(I) \leq c(s)$ for all $s \in I$, we have for any $s \in I$ that
$$\int_{0}^{\frac{2\pi}{s}} \cos(sv) \sqrt{1 - \frac{\cosh (v +t_{k-1})}{\cosh t_k}} \, dv < -2c(s)+c(I) \leq -c(I).$$
Multiplying both sides by $2\sqrt{2}$ yields the claim by the definition of $h_{t_k}(s)$.
\end{proof}

Now we are ready to prove Proposition \ref{prop:spectral}.

\begin{proof}[Proof of Proposition \ref{prop:spectral}]
Let $T > 0$ and $s\in I$.
By Lemma \ref{prop:spectralperiodbound}, we know that for any $k \geq k(I)$,
$$|h_{t_k}(s)| \geq 2c(I).$$
We thus have a sequence of times $(t_k)_{k \geq k_0(I)}$ on which $|h_{t_k}(s)|$ is bounded uniformly from below. Using the uniform Lipschitz continuity of $t \mapsto |h_{t}(s)|$ on $(1,+\infty)$ we can find intervals around each of these points on which the function is also uniformly bounded from below, so that we have a lower bound on a set of positive measure. More precisely, we can find an interval $J$ around $0$ and $k_1  \geq k_0(I)$ such that for any $k \geq k_1$, $|h_t(s)| \geq c(I)$ on the translated interval $J + t_k$.
Denote
$$A(s,T) = \bigcup_{k = k_1}^N J+t_k,$$
where $N$ is chosen such that 
$$ 2\pi N/s < T \leq 2\pi (N+1)/s.$$
Note that
$$|A(s,T)| \geq (N - k_1) |J| > 0,$$
and we have in particular
$$\frac{1}{T} \int_0^{T} |h_t(s)|^2 \, dt \geq  \left(\frac{N - k_1}{T} \right) |J| \, c(I)$$
and using that $ \frac{N}{T} \geq \frac{s}{4\pi} $, we obtain
$$\frac{1}{T} \int_0^{T} |h_t(s)|^2 \, dt \geq  \left(\frac{s}{4\pi} - \frac{k_1}{T} \right) |J| \, c(I).$$
Thus for $T_I$ large enough there exists a constant $C_I > 0$ such that for any $T \geq T_I$
$$\frac{1}{T} \int_0^{T} |h_t(s)|^2 \, dt \geq  C_I,$$
which completes the proof.
\end{proof}

\section{Asymptotic number of eigenvalues} \label{sec:spectralBS}

In this section, we prove an analog of Weyl's law for eigenvalues in a fixed bounded interval and sequences of surfaces $X_n = \Gamma_n \setminus \H$ converging in the sense of Benjamini-Schramm to $\H$. This allows us to normalise by the number of eigenvalues in Theorem \ref{thm:sequences}.  This estimate of the number of eigenvalues is a consequence of the results of \cite{DGW78,DGW79} in the case of sequences of coverings, and is proved in \cite{Hu82} for sequences of surfaces with increasing injectivity radius. A more general version was given in \cite{ABBGNRSPreprint} for Benjamini-Schramm convergence. Here we assume that the injectivity radii are uniformly bounded from below: $\InjRad(X_n) \geq \ell_{\min}$ for all $n \in \N$ for some $\ell_{\min} > 0$, which we also assume in Theorem \ref{thm:sequences}.

\begin{lemma}\label{lma:boundingeigenvalues}
For any compact interval $I \subset (1/4,+\infty)$, we have
	$$\lim_{n \to \infty}\frac{N(X_n,I)}{\Vol(X_n)} = \frac{1}{4\pi} \int_\R \chi_I(1/4 +\rho^2) \tanh(\pi \rho) \rho \, d\rho,$$
where $N(X_n,I)$ is the number of eigenvalues in the interval $I$.
\end{lemma}

The proof of Lemma \ref{lma:boundingeigenvalues} follows immediately from the following statement, by approximating the characteristic function of the interval $I$ by continuous functions and using the dominated convergence theorem.

\begin{thm}\label{thm:generalcontinuous}
Let $f : [0,+\infty) \to \R$ be a compactly supported continuous function. Then
\begin{align*}
\lim_{n \to \infty}\frac1{\Vol(X_n)} \sum_{j = 0}^\infty f(\lambda_j^{(n)})  &= \frac{1}{4\pi} \int_\R f(1/4 +\rho^2) \tanh(\pi \rho) \rho \, d\rho,
\end{align*}
where $\lambda_0^{(n)} = 0 < \lambda_1^{(n)} \leq \lambda_2^{(n)} \leq \dots$ are the eigenvalues of the Laplacian on $X_n$.
\end{thm}

To establish Theorem \ref{thm:generalcontinuous}, we use the Selberg pre-trace formula. For a proof of this theorem and a more general treatment of the Selberg trace formula, see for example \cite{Mar12, Iwa02, Ber16}.
Given a function $h : \C \to \C$ satisfying the regularity assumptions of Theorem \ref{lma:selbergtraceformula}, and its Fourier transform $\hat h$, the invariant kernel $k(z,z') = k_h(z,z')$ associated to $h$ is defined by
$$k(z,z') = k_h(z,z') := -\frac1{\sqrt{2}\pi} \int_{d(z,z')}^\infty \frac{(\hat h)'(\rho)}{\sqrt{\cosh \rho - \cosh d(z,z')}} \, d\rho.$$
This kernel is simply obtained from the inverse Selberg transform of $h$.

\begin{thm}[Selberg pre-trace formula] \label{lma:selbergtraceformula} Let $X$ be a compact hyperbolic surface. Denote by $(\psi_j)$ an orthonormal basis of eigenfunctions of the Laplacian in $L^2(X)$ and by $(\lambda_j)$ the corresponding non decreasing sequence of eigenvalues. Suppose $h : \C \to \C$ satisfies:
\begin{itemize}
\item[(1)] $h$ is analytic on the strip $|\Im z| \leq \sigma$ for some $\sigma > 1/2$;
\item[(2)] $h$ is even, that is, $h(-z) = h(z)$;
\item[(3)] $|h(z)| \lesssim (1+|\Re z|)^{-2-\delta}$
for some fixed $\delta > 0$, uniformly for all $z$ in the strip $|\Im z| \leq \sigma$.
\end{itemize}
Then the following formula converges absolutely:
$$\sum_{j = 0}^\infty h(s_j) |\psi_j(z)|^2 = \frac{1}{4\pi} \int_\R h(\rho) \tanh(\pi \rho) \rho \, d\rho + \sum_{\gamma \in \Gamma(X) - \{\id\}} k(z, \gamma\cdot z)$$
uniformly in $z \in \H$, where $k(z,z')$ is the invariant kernel associated to $h$, and $s_j$ is defined by the relation $\lambda_j = 1/4 + s_j^2$.
\end{thm}

We will apply this formula to the heat kernel, that is the kernel $k_{h_t}(z,z') = p_t(d(z,z'))$ associated with the Selberg transform $h_t(s) = e^{-t(1/4+s^2)}$. We need the following heat kernel estimate, see for example Buser's book \cite[Lemma 7.4.26]{Bus10}:
\begin{lemma}\label{l:heat decay}
For any $ t \geq 0$ there is a constant $C_t > 0$ such that
$$0 \leq p_t(\rho) \leq C_t e^{-\rho^2}.$$
\end{lemma}
\begin{prop}\label{p:heat}
Let $t > 0$. Then
\begin{align*}
\lim_{n \to \infty}\frac1{\Vol(X_n)} \sum_{j = 0}^\infty e^{-t\lambda_j^{(n)}}  &= \frac{1}{4\pi} \int_\R e^{-t(1/4 +\rho^2)} \tanh(\pi \rho) \rho \, d\rho 
\end{align*}
\end{prop}

\begin{proof}
Let $D_n$ be a fundamental domain for $X_n$. Write
$$X_n(R) := \{z \in D_n : \InjRad_{X_n}(z) \leq R\}.$$
 We apply Selberg's pre-trace formula (Theorem \ref{lma:selbergtraceformula}) to the surface $X_n$ and the function $h_t$ to obtain
$$\sum_{j = 0}^\infty e^{-t(1/4+s_j^2)} |\psi_j(z)|^2 = \frac{1}{4\pi} \int_\R e^{-t(1/4+ \rho^2)} \tanh(\pi \rho) \rho \, d\rho + \sum_{\gamma \in \Gamma_n - \{\id\}} p_t(d(z, \gamma\cdot z)).$$
where $s_j = s_j^{(n)}$ is defined by the relation $\lambda_j^{(n)} = 1/4 + s_j^2$. Note that by Lemma \ref{l:heat decay}, the last sum is bounded in absolute value by
$$ C_t \sum_{\gamma \in \Gamma_n - \{\id\}} e^{-d(z,\gamma z)^2}  \lesssim_t  \sum_{k = k_n(z)}^{+\infty} \frac{e^k}{\ell_{\min}} e^{-k^2},$$
where $k_n(z) = [\InjRad_{X_n}(z)]$ and we used as in the proof of Lemma \ref{l:normsurface} that for any $z \in \H$
$$ \# \{\gamma \in \Gamma_n: d(z,\gamma z) \leq R \} = O(e^R/\ell_{\min}).$$

We integrate this expression over $D_n$. We obtain
\begin{align*}
\sum_{j = 0}^\infty e^{-t(1/4+s_j^2)}  &= \frac{\Vol(X_n)}{4\pi} \int_\R e^{-t(1/4+ \rho^2)} \tanh(\pi \rho) \rho \, d\rho + O\left( \int_{D_n}  \sum_{k = k_n(z)}^{+\infty} \frac{e^k}{\ell_{\min}} e^{-k^2}\, d\mu(z) \right)\\
\end{align*}
The last term can be split into an integral over the points in $X_n(R)$ and the complement $X_n(R)^c$ in $D_n$. We use the fact that the sum converges for the first term and that $\InjRad_{X_n}(z) > R$ for the second, to obtain
$$ \int_{D_n}  \sum_{k = k_n(z)}^{+\infty} \frac{e^k}{\ell_{\min}} e^{-k^2}\, d\mu(z)
\lesssim  \frac{\Vol(\{ z \in X_n : \InjRad_{X_n}(z) < R\})}{\ell_{\min}} + \frac{\Vol(X_n)}{\ell_{\min}} e^{-R^2}. $$
This implies that as $n \to \infty$ we have
\begin{align*}\frac{1}{\Vol(X_n)}\sum_{j = 0}^\infty e^{-t(1/4+s_j^2)} &=  \frac{1}{4\pi} \int_\R e^{-t(1/4+ \rho^2)} \tanh(\pi \rho) \rho \, d\rho \\
&\quad+ O(e^{-R^2}/{\ell_{\min}}) +  O\left( \frac1{\ell_{\min}}\frac{\Vol(\{ z \in X_n : \InjRad_{X_n}(z) < R\})}{\Vol(X_n)}\right).\end{align*}
Now using the Benjamini-Schramm convergence of $X_n$ to $\H$, we can fix a sequence $R_n \to \infty$ such that
$$\frac{\Vol(\{ z \in X_n : \InjRad_{X_n}(z) < R_n\})}{\Vol(X_n)} \to 0$$
as $n \to \infty$. Hence applying the above for $R = R_n$ and letting $n \to \infty$ gives the result.
\end{proof}

Using Proposition \ref{p:heat}, we follow a similar argument to \cite{Do83} to establish Theorem \ref{thm:generalcontinuous}.

\begin{proof}[Proof of Theorem \ref{thm:generalcontinuous}]
Write
$$S_n(f) := \sum_{j = 0}^\infty f(\lambda_j^{(n)}) \quad \text{and} \quad I(f) := \frac{1}{4\pi} \int_\R f(1/4 +\rho^2) \tanh(\pi \rho) \rho \, d\rho$$
With this notation, we need to prove that for $f$ a continuous compactly supported function
$$\frac{S_n(f)}{\Vol(X_n)}  \to I(f).$$
Define $g(x) := f(x)e^{x}$. As $f$ is continuous and compactly supported, so is $g$, and it can be approximated uniformly by linear combinations of exponential functions $e^{-tx}$, $t > 0$ (see for example Lemma 3.1 from \cite{Do83}). For any $\eps >0$ we may choose $g_\eps$ of the form
$$g_\eps(x) = \sum_{k} a_k e^{-t_k x}$$
where the sum is finite such that $\|g_\eps - g\|_\infty < \eps$.

For a fixed $n \in \N$, estimate
\begin{align*}\Big|\frac{S_n(f)}{\Vol(X_n)} - I(f)\Big|  \leq \,\, &\Big|\frac{S_n(f)}{\Vol(X_n)} - \frac{S_n(g_\eps e^{-x})}{\Vol(X_n)}\Big|\\
& + \Big|\frac{S_n(g_\eps e^{-x})}{\Vol(X_n)}-I(g_\eps e^{-x})\Big| \\
& + \Big|I(g_\eps e^{-x})-I(f)\Big|
\end{align*}
The first term has a bound
$$\Big|\frac{S_n(f)}{\Vol(X_n)} - \frac{S_n(g_\eps e^{-x})}{\Vol(X_n)}\Big| \leq \|g-g_\eps\|_\infty \frac{S_n(e^{-x})}{\Vol(X_n)} < \eps \frac{S_n(e^{-x})}{\Vol(X_n)}$$
and the second term
$$\Big|\frac{S_n(g_\eps e^{-x})}{\Vol(X_n)}-I(g_\eps e^{-x})\Big| \leq \sum_{k}|a_k| \Big|\frac{S_n(e^{-(t_k+1)x})}{\Vol(X_n)}-I(e^{-(t_k+1)x})\Big|.$$
Thus letting $n \to \infty$ and applying Proposition \ref{p:heat} with $t = 1$ and $t = t_k+1$ for every $k$, we obtain
$$\limsup_{n \to \infty} \Big|\frac{S_n(f)}{\Vol(X_n)} - I(f)\Big| \leq \eps I(e^{-x}) + |I(g_\eps e^{-x}) - I(f)|.$$
The right-hand side converges to $0$ as $\eps \to 0$ by the dominated convergence theorem as $g_\eps(x) e^{-x} \to f(x)$ uniformly.
\end{proof}

\section*{Acknowledgements} 

We thank Yves Colin de Verdi\`{e}re for suggesting the problem. We are grateful to Nalini Anantharaman, Elon Lindenstrauss and Nicolas de Saxc\'e for useful discussions and comments. Moreover, we thank the enthusiastic community in the School of Mathematics at the University of Bristol, where the majority of the research was carried out, in particular, Alex Gorodnik, Thomas Jordan, Jens Marklof, Abhishek Saha, Roman Schubert and Corinna Ulcigrai for several helpful conversations about this work. We are particularly indebted to Roman Schubert for spotting a mistake in an early draft. We also thank Fr\'ed\'eric Naud, Zeev Rudnick, Steve Zelditch and the anonymous referees for useful comments on earlier versions of the manuscript.

\bibliographystyle{plain}

\end{document}